\documentclass[11pt,a4paper]{amsart}

\usepackage[T1]{fontenc}
\usepackage[utf8]{inputenc}
\usepackage{lmodern}
\usepackage{a4wide}
\usepackage{amsmath,amssymb,mathtools}
\usepackage{microtype}
\usepackage{booktabs}
\usepackage{enumitem}
\usepackage{url}
\usepackage[
 colorlinks=true,
 linkcolor=blue,
 citecolor=blue,
 urlcolor=blue,
 pdftitle={A new bound for the Kalton--Roberts constant},
 pdfauthor={Boon Suan Ho and Tomasz Kania}
]{hyperref}

\newtheorem{theorem}{Theorem}[section]
\newtheorem{lemma}[theorem]{Lemma}
\newtheorem*{theoremA}{Theorem A}

\DeclareMathOperator{\dist}{dist}

\newcommand{\cA}{\mathcal A}
\newcommand{\cB}{\mathcal B}
\newcommand{\cC}{\mathcal C}
\newcommand{\cF}{\mathcal F}
\newcommand{\cL}{\mathcal L}
\newcommand{\KR}{K_{\mathrm{KR}}}
\newcommand{\RR}{\mathbb R}

\title[A new bound for the Kalton--Roberts constant]{A new bound for the Kalton--Roberts constant}

\author[B.~S.~Ho]{Boon Suan Ho}
\address[B.~S.~Ho]{Department of Mathematics\\
National University of Singapore\\
Singapore}
\email{hbs@u.nus.edu}

\author[T.~Kania]{Tomasz Kania}
\address[T.~Kania]{Mathematical Institute\\Czech Academy of Sciences\\
\v Zitn\'a 25 \\115 67 Praha 1\\Czech Republic  and  Institute of Mathematics and Computer Science\\ Jagiellonian University\\ {\L}ojasiewicza 6, 30-348 Krak\'{o}w, Poland
}
\email{kania@math.cas.cz, tomasz.marcin.kania@gmail.com}
\thanks{RVO: 67985840.}

\subjclass[2020]{Primary 28A10; Secondary 05C80, 28A12, 41A65}
\keywords{Kalton--Roberts constant, nearly additive set function,
finitely additive measure, best uniform approximation, expander graph}

\begin{document}

\begin{abstract}
Let \(\KR\) be the least constant governing uniform approximation of
approximately additive real-valued set functions by finitely additive signed
measures.  We prove that
\[
        \KR\leqslant \frac{374167}{20000}=18.70835.
\]
Starting from a norming functional for best uniform approximation, we form
two families of almost extremal sets with complementary point frequencies.
Intersections and biregular expanders then give two recombination schemes.
The sharper estimate comes from using the schemes on different ranges and
from cancelling the positive and negative errors with unequal weights.  This
improves the previously published bound \(38.8\).
\end{abstract}

\maketitle

\section{Introduction}

Let \(\Omega\) be a set and let \(\cF\) be an algebra of subsets of
\(\Omega\).  A function \(f\colon\cF\to\RR\) is called
\(\Delta\)-additive if \(f(\varnothing)=0\) and
\begin{equation}\label{eq:additivity}
 \lvert f(A)+f(B)-f(A\cup B)\rvert\leqslant\Delta
 \qquad(A,B\in\cF,\ A\cap B=\varnothing).
\end{equation}
The theorem of Kalton and Roberts asserts that there is a universal constant
\(C\) for which every \(\Delta\)-additive function is within \(C\Delta\),
uniformly on \(\cF\), of a finitely additive signed measure.  We denote the
least such constant by \(\KR\).

Kalton and Roberts proved that \(\KR\leqslant89/2\)
\cite{KaltonRoberts}.  Bondarenko, Prymak and Radchenko subsequently reduced
this to \(38.8\) \cite{BondarenkoPrymakRadchenko}.  Feige, Feldman and
Talgam-Cohen obtained the related bound \(23.811\) for weakly modular set
functions approximated by linear set functions, where a constant term is
allowed \cite{FeigeFeldmanTalgamCohen}.  The best lower bound for the present
constant is \(3\), due to Gnacik, Guzik and Kania
\cite{GnacikGuzikKania}; it improved Pawlik's earlier bound \(3/2\)
\cite{Pawlik}.

Our result is as follows.

\begin{theoremA}
For every set algebra \(\cF\), every \(\Delta\geqslant0\), and every
\(\Delta\)-additive function \(f\colon\cF\to\RR\), there is a finitely
additive signed measure \(\mu\) on \(\cF\) such that
\[
        \sup_{A\in\cF}\lvert f(A)-\mu(A)\rvert
        \leqslant \frac{374167}{20000}\,\Delta.
\]
Consequently, \(\KR\leqslant374167/20000=18.70835\).
\end{theoremA}

As in the proof of Kalton and Roberts, the final step is a combinatorial
recombination.  The families to be recombined arise here from a norming
functional for best approximation by additive functions.  Its positive and
negative parts give two families of almost extremal sets with controlled point
frequencies.  Intersections reduce those frequencies, after which three
biregular expanders can be applied.  We use one direct two-step chain
and a second chain in which an extra intersection is inserted between two
applications of the same expander.  Neither chain is best over the whole
range; using them on different subintervals is what yields the stated bound.

The final saving comes from an asymmetric combination of the two families.
Let \(p\) and \(q\) be the masses of the positive and negative parts of the
norming functional, with \(q\leqslant p\), and put \(u=f(U)\) after the
finite reduction.  The deficit estimate contains a
term proportional to \(q(1-u)\), whereas the surplus estimate contains one
proportional to \(p(1+u)\).  Rather than averaging these estimates with equal
weights, we choose the weights so that the two occurrences of \(u\) cancel.

\section{Finite reduction and extremal families}

We begin with the standard reduction to finite power-set algebras.

\begin{lemma}\label{lem:finite-reduction}
Suppose that the assertion of Theorem~A holds on every finite power-set
algebra, with a constant \(C\) in place of \(374167/20000\).  Then it holds with
the same constant on every set algebra.
\end{lemma}

\begin{proof}
The case \(\Delta=0\) is immediate.  By scaling, it remains to consider
\(\Delta=1\).  For \(A\in\cF\), let
\(I_A=[f(A)-C,f(A)+C]\), and consider the compact space
\[
        X=\prod_{A\in\cF}I_A.
\]
For disjoint \(A,B\in\cF\), the equation
\(\nu(A\cup B)=\nu(A)+\nu(B)\) defines a closed subset of \(X\).  Any
finite collection of these equations involves only a finite subalgebra
\(\cF_0\) of \(\cF\).  After identifying \(\cF_0\) with the power set of
its atoms, the finite assertion gives an additive measure on \(\cF_0\) which
lies in the required intervals.  Extending its remaining coordinates
arbitrarily gives a point of \(X\) satisfying the chosen equations.  The
closed conditions therefore have the finite intersection property.
Compactness supplies a point of \(X\) satisfying all of them, and this point
is the required finitely additive measure.
\end{proof}

For the rest of the proof, \(U\) is finite and \(f\colon2^U\to\RR\) is
\(1\)-additive.  Additive functions on \(2^U\) form the subspace
\[
 \cL=\left\{A\longmapsto\sum_{i\in A}a_i:(a_i)_{i\in U}\in\RR^U\right\}
 \subseteq\RR^{2^U}.
\]
Choose a closest element of \(\cL\) to \(f\) in the supremum norm and
subtract it.  Thus, writing
\[
        M=\lVert f\rVert_\infty,
\]
we have \(M=\dist_\infty(f,\cL)\).

If \(M=0\), there is nothing left to prove.  We therefore assume from now on
that \(M>0\).

\begin{lemma}\label{lem:norming-vector}
There is a rational vector \(\lambda=(\lambda_A)_{A\subseteq U}\) such
that
\begin{enumerate}[label=\textup{(\roman*)}]
\item \(\sum_A\lvert\lambda_A\rvert=1\);
\item \(\sum_{A\ni i}\lambda_A=0\) for every \(i\in U\);
\item \(\lambda_A>0\) implies \(f(A)=M\), and \(\lambda_A<0\) implies
      \(f(A)=-M\).
\end{enumerate}
\end{lemma}

\begin{proof}
The annihilator of \(\cL\) consists of the vectors satisfying the equations
in (ii).  The finite-dimensional Hahn--Banach criterion for best uniform
approximation gives a vector in this annihilator with \(\ell_1\)-norm one and
\(\langle\lambda,f\rangle=M\).  Equality in H\"older's inequality gives
(iii).

It remains to ensure rationality.  Put
\[
 P=\{A:f(A)=M\},\qquad N=\{A:f(A)=-M\}.
\]
Use non-negative variables \(x_A\) for \(A\in P\) and \(y_A\) for
\(A\in N\), subject to
\[
 \sum_{A\in P}x_A+\sum_{A\in N}y_A=1,
 \qquad
 \sum_{\substack{A\in P\\ i\in A}}x_A
 =\sum_{\substack{A\in N\\ i\in A}}y_A \quad(i\in U).
\]
The preceding norming vector shows that this bounded rational polytope is
non-empty.  A vertex has rational coordinates.  Taking \(\lambda_A=x_A\)
on \(P\), \(\lambda_A=-y_A\) on \(N\), and zero elsewhere proves the
claim.
\end{proof}

Both signs occur in \(\lambda\).  Indeed, if all its coordinates had the same sign, the
zero-marginal equations would force \(\lambda\) to be supported on
\(\varnothing\), which is impossible because \(f(\varnothing)=0\neq M\).
Let
\[
 p=\sum_{\lambda_A>0}\lambda_A,
 \qquad
 q=-\sum_{\lambda_A<0}\lambda_A.
\]
After replacing \(f\) by \(-f\), if necessary, we may assume
\begin{equation}\label{eq:pq}
        p+q=1,\qquad 0<q\leqslant p,\qquad q\leqslant\frac12.
\end{equation}
Put \(u=f(U)\).  If \(A_+\) is a positive active set, then
\(f(A_+)=M\), \(f(A_+^c)\geqslant-M\), and approximate additivity gives
\(u\geqslant-1\).  A negative active set \(A_-\) gives
\(u\leqslant1\).  Hence
\begin{equation}\label{eq:u}
        -1\leqslant u\leqslant1.
\end{equation}

A rational weighted family will mean a finitely supported probability
distribution on \(2^U\) with rational weights.  Its point frequency at
\(i\in U\) is the total weight of the sets containing \(i\).  Clearing
denominators turns such a family into a finite multiset without changing
frequencies or averages.

\begin{lemma}\label{lem:extremal-families}
There are rational weighted families \(\cA\) and \(\cB\) such that
\begin{enumerate}[label=\textup{(\roman*)}]
\item every point has frequency \(q\) in \(\cA\), and the average of
      \(M-f(A)\) over \(\cA\) is at most \(q(1-u)\);
\item every point has frequency \(p\) in \(\cB\), and the average of
      \(M+f(B)\) over \(\cB\) is at most \(p(1+u)\).
\end{enumerate}
\end{lemma}

\begin{proof}
Write \(\lambda=\lambda^+-\lambda^-\), and for \(i\in U\) put
\[
 m_i^+=\sum_{\substack{A\ni i}}\lambda_A^+,
 \qquad
 m_i^-=\sum_{\substack{A\ni i}}\lambda_A^-.
\]
The zero-marginal equations say that \(m_i^+=m_i^-\).

Form \(\cA\) from the positive active sets with their \(\lambda^+\)-weights
and the complements of the negative active sets with their
\(\lambda^-\)-weights.  Its total weight is one, and the frequency of \(i\)
is
\[
        m_i^++(q-m_i^-)=q.
\]
Positive active sets have deficit zero.  If \(A_-\) is negative active, then
\(A_-\sqcup A_-^c=U\), and \(1\)-additivity gives
\[
        M-f(A_-^c)\leqslant1-u.
\]
Averaging gives (i).  Define \(\cB\) using the negative active sets with
their \(\lambda^-\)-weights and the complements of the positive active sets
with their \(\lambda^+\)-weights.  The same calculation gives point
frequency \(p\).
The inequality
\[
        M+f(A_+^c)\leqslant1+u
\]
for positive active \(A_+\) gives (ii).
\end{proof}

\section{Intersections and recombination}

The next estimate is the only use of approximate additivity before the
expander argument.

\begin{lemma}\label{lem:modularity}
For all \(A,B\subseteq U\),
\begin{equation}\label{eq:modularity}
 \lvert f(A)+f(B)-f(A\cup B)-f(A\cap B)\rvert\leqslant2.
\end{equation}
If \(f(A)=M-a\) and \(f(B)=M-b\), then
\(f(A\cap B)\geqslant M-a-b-2\).  The analogous assertion holds for
surpluses after replacing \(f\) by \(-f\).
\end{lemma}

\begin{proof}
Put \(I=A\cap B\), \(X=A\setminus B\), and \(Y=B\setminus A\).  Apply
\eqref{eq:additivity} to \(A=I\sqcup X\) and to
\(A\cup B=B\sqcup X\).  Subtracting the corresponding identities with
their error terms and applying the triangle inequality gives
\eqref{eq:modularity}.  Since \(f(A\cup B)\leqslant M\), the deficit
estimate follows; the surplus estimate is the same argument for \(-f\).
\end{proof}

For a weighted family \(\cC\), let \(\cC^{\cap\ell}\) be the product
distribution of intersections of \(\ell\) independently chosen members.  For
\(\tau\in\mathbb Q\cap[0,1]\), write
\[
 \cC^{(\ell,\tau)}
 =(1-\tau)\cC^{\cap\ell}+\tau\cC^{\cap(\ell+1)}.
\]

\begin{lemma}\label{lem:intersections}
Suppose that the point frequencies of \(\cC\) are at most \(t\), and its
average deficit is at most \(E\).  Put \(m=\ell+\tau\).  Then
\(\cC^{(\ell,\tau)}\) has point frequencies at most
\begin{equation}\label{eq:intersection-frequency}
        (1-\tau)t^\ell+\tau t^{\ell+1}
\end{equation}
and average deficit at most
\begin{equation}\label{eq:intersection-cost}
        mE+2(m-1).
\end{equation}
The same conclusion holds for surpluses.
\end{lemma}

\begin{proof}
Point frequencies multiply under independent intersections.  Repeated use of
Lemma~\ref{lem:modularity} shows that an \(\ell\)-fold intersection has deficit
at most the sum of the source deficits plus \(2(\ell-1)\).  Averaging the
\(\ell\)- and \((\ell+1)\)-fold estimates gives
\eqref{eq:intersection-cost}.
\end{proof}

Let \(r\geqslant3\) and let \(c>r\) be integers, and put
\(\theta=r/c\).  An \((\alpha,r,\theta)\)-expander is a bipartite
multigraph \(G=(V,W;E)\) such that, for a positive integer \(N\) with
\(\theta N\) integral,
\[
 |V|=N,\qquad |W|=\theta N,
\]
every left vertex has degree \(r\), every right vertex has degree \(c\), and
every set of at most \(\alpha N\) left vertices has at least as many distinct
neighbours.  We call such \(N\) admissible.  We say that these expanders
exist if they exist for every sufficiently large admissible \(N\).

\begin{lemma}\label{lem:recombination}
Suppose that \((\alpha,r,\theta)\)-expanders exist.  Let \(\cC\) be a
rational weighted family with point frequencies at most \(\alpha\) and
average deficit \(D\).  There is a recombined family with point frequencies
at most \(\alpha/\theta\).  If its average deficit is \(D'\), then
\begin{equation}\label{eq:recombination}
        (1-\theta)M\leqslant D-\theta D'+2r-1-\theta.
\end{equation}
The same assertion holds for surpluses.
\end{lemma}

\begin{proof}
Clear the denominators in \(\cC\).  Choose a sufficiently large admissible
\(N\) which is also a multiple of the resulting multiset cardinality, and
duplicate the multiset to have \(N\) members.  Label the left vertices by
these sets, say \((C_v)_{v\in V}\).

For \(i\in U\), the left vertices whose sets contain \(i\) form a set of
cardinality at most \(\alpha N\).  The expansion property and Hall's theorem
give a matching of these occurrences into distinct right neighbours.  Place
\(i\) in the edge piece used by this matching.  Doing this separately for
every \(i\) partitions each \(C_v\) into its incident edge pieces, while the
pieces arriving at a fixed right vertex are pairwise disjoint.  Their union
will be denoted by \(T_w\).  The point frequencies of the family
\((T_w)_{w\in W}\) are at most \(\alpha/\theta\).

Let \(d_v=M-f(C_v)\) and \(d'_w=M-f(T_w)\).  Iterated
\(1\)-additivity and biregularity give
\[
 \sum_{e\ni v}f(C_e)\geqslant f(C_v)-(r-1),
 \qquad
 \sum_{e\ni w}f(C_e)\leqslant f(T_w)+(c-1).
\]
Summing first over \(V\) and then over \(W\), and using
\(c\theta=r\), yields
\[
 N(M-D)-N(r-1)
 \leqslant \theta N(M-D')+N(r-\theta).
\]
This is \eqref{eq:recombination}.  Apply the same argument to \(-f\) for
surpluses.
\end{proof}

\section{Biregular expanders}

We use a standard random biregular construction in the form below.  The proof
is included because the numerical parameters matter in the final estimate.
For \(a>b>0\), set
\[
 h(a,b)=a\log a-b\log b-(a-b)\log(a-b),
\]
and, when \(\theta=r/c\), define
\begin{equation}\label{eq:Phi}
 \Phi_{r,\theta}(x)
 =h(1,x)+h(\theta,x)+h(rx/\theta,rx)-h(r,rx).
\end{equation}

\begin{lemma}\label{lem:pippenger}
Let \(0<\alpha<\theta<1\), \(r\geqslant3\), and
\(\theta=r/c\) for an integer \(c>r\).  Suppose that some
\(\delta\in(0,\alpha)\) satisfies
\begin{equation}\label{eq:pippenger-small}
        e^2\theta^{1-r}\delta^{r-2}<\frac1{20}
\end{equation}
and
\begin{equation}\label{eq:pippenger-mid}
        \Phi_{r,\theta}(x)<0\qquad(\delta\leqslant x\leqslant\alpha).
\end{equation}
Then \((\alpha,r,\theta)\)-expanders exist.
\end{lemma}

\begin{proof}
Let \(N\) be sufficiently large and admissible.  Take \(r\) copies of each
of \(N\) left vertices and \(c\) copies of each of \(\theta N\) right
vertices, and choose a uniformly random perfect matching between the two sets
of copies.  The resulting multigraph is \((r,c)\)-biregular.

Fix \(m=xN\) left vertices.  If their neighbourhood has fewer than \(m\)
vertices, it is contained in some set of \(m\) right vertices.  For a fixed
choice of the left and right sets, the probability of this containment is
\[
        \frac{\binom{cm}{rm}}{\binom{rN}{rm}}.
\]
Consequently, the expected number of bad left sets of size \(m\) is at most
\begin{equation}\label{eq:bad-expectation}
 \binom Nm\binom{\theta N}m
 \frac{\binom{cm}{rm}}{\binom{rN}{rm}}.
\end{equation}

For \(\delta N\leqslant m\leqslant\alpha N\), the usual entropy estimates
for binomial coefficients \cite[Appendix~A]{AlonSpencer} turn
\eqref{eq:bad-expectation} into
\[
        (rN+1)\exp\bigl(N\Phi_{r,\theta}(x)\bigr).
\]
By \eqref{eq:pippenger-mid} and compactness, the sum of these quantities tends
to zero.

For \(m=xN\leqslant\delta N\),
\[
 \frac{\binom{cm}{rm}}{\binom{rN}{rm}}
 \leqslant\left(\frac{m}{\theta N}\right)^{rm}.
\]
Together with \(\binom nj\leqslant(en/j)^j\), this bounds
\eqref{eq:bad-expectation} by
\[
        \left(e^2\theta^{1-r}x^{r-2}\right)^m.
\]
The sum over all \(m\geqslant1\) is at most \(1/19\) by
\eqref{eq:pippenger-small}.  Thus the expected number of bad sets is less
than one for large admissible \(N\), proving the assertion.  This is the
random biregular argument used in Pippenger's work on superconcentrators
\cite{Pippenger}.
\end{proof}

Put
\begin{equation}\label{eq:caps}
        \beta=\frac{211}{2100},
        \qquad
        \gamma=\frac{49}{400}.
\end{equation}
We shall use the following three rows:
\begin{center}
\begin{tabular}{c c c c c}
\toprule
 & \(\alpha\) & \(r\) & \(c\) & \(\theta\)\\
\midrule
\(E_1\) & \(211/2100\) & \(4\) & \(12\) & \(1/3\)\\
\(E_2\) & \(211/700\)  & \(4\) & \(7\)  & \(4/7\)\\
\(E_3\) & \(49/400\)   & \(4\) & \(11\) & \(4/11\)\\
\bottomrule
\end{tabular}
\end{center}

Here are exact checks of the hypotheses of Lemma~\ref{lem:pippenger}, with
\(\delta=1/100\).  We use \(e^2<739/100\).
\begin{center}
\small
\begin{tabular}{c c c c c}
\toprule
 & \((739/100)\theta^{-3}\delta^2\) &
 \(\Phi(\delta)\) & \(\Phi(\alpha)\) & lower bound for \(\Phi''\)\\
\midrule
\(E_1\) & \(19953/10^6\) & \(<-1/1000\) & \(<-1/1000\) &
\(640079/34393\)\\
\(E_2\) & \(253477/64000000\) & \(<-1/1000\) & \(<-1/20000\) &
\(33791/5697\)\\
\(E_3\) & \(983609/64000000\) & \(<-1/1000\) & \(<-1/100000\) &
\(789167/51989\)\\
\bottomrule
\end{tabular}
\end{center}
The entries in the second column are smaller than \(1/20\), and those in the
last column are greater than \(3\).  To explain the remaining arithmetic,
differentiation gives
\begin{equation}\label{eq:Phi-second}
 \Phi''_{r,\theta}(x)
 =\frac{r-2}{x}+\frac{r-1}{1-x}-\frac1{\theta-x}.
\end{equation}
On \([\delta,\alpha]\), the right-hand side is at least
\[
 \frac{r-2}{\alpha}+(r-1)-\frac1{\theta-\alpha},
\]
which gives the last column.  Hence \(\Phi\) is convex, so its maximum on
the interval is attained at an endpoint.

The endpoint logarithms may be bounded without decimal arithmetic.  Given a
positive rational \(z\), choose an integer \(k\) for which
\(y=2^kz\in[1,2]\).  Then \(\log z=\log y-k\log2\), and, with
\(t=(y-1)/(y+1)\),
\[
 \log y=2\sum_{n=0}^{L-1}\frac{t^{2n+1}}{2n+1}+R_L,
 \qquad
 0\leqslant R_L\leqslant
 \frac{2t^{2L+1}}{(2L+1)(1-t^2)}.
\]
Here \(0\leqslant t\leqslant1/3\).  Substitution with \(L=30\) gives the
rational upper bounds displayed above.  Thus all three expander
families exist.

\section{Proof of Theorem A}

We retain the notation of Section~2.  In particular, \(q\leqslant p=1-q\),
\(-1\leqslant u\leqslant1\), and the families \(\cA,\cB\) are supplied by
Lemma~\ref{lem:extremal-families}.

We first record two consequences of the preceding section.  Suppose that a
weighted family has point frequencies at most \(\beta\) and average deficit
\(D_0\).  Apply \(E_1\), and call the new average deficit \(D_1\).  The new
frequencies are at most \(3\beta=211/700\), so \(E_2\) may be applied; write
\(D_2\) for the final average deficit.  Lemma~\ref{lem:recombination} gives
\[
 \frac23M\leqslant D_0-\frac13D_1+\frac{20}{3},
 \qquad
 \frac37M\leqslant D_1-\frac47D_2+\frac{45}{7}.
\]
Adding one third of the second inequality to the first, and discarding the
non-positive term containing \(D_2\) (recall that \(D_2\geqslant0\)), yields
\begin{equation}\label{eq:direct-estimate}
        \frac{17}{21}M\leqslant D_0+\frac{185}{21}.
\end{equation}
The same estimate holds with deficit replaced by surplus.

There is another useful way to combine expanders.  Set
\[
        \theta=\frac4{11},\qquad
        K=2\cdot4-1-\theta=\frac{73}{11}.
\]
Starting with point frequencies at most \(\gamma\), one application of
\(E_3\) produces frequencies at most
\[
        t=\frac{\gamma}{\theta}=\frac{539}{1600}.
\]
Take the convex combination of the resulting family and its
pairwise-intersection family, assigning the latter weight
\[
        \sigma=\frac{t-\gamma}{t-t^2}=\frac{11200}{11671},
 \qquad
        m=1+\sigma=\frac{22871}{11671}.
\]
The function \(s\mapsto(1-\sigma)s+\sigma s^2\) is increasing for
\(s\geqslant0\), and its value at \(t\) is \(\gamma\).  The resulting
family therefore has point frequencies at most \(\gamma\), so \(E_3\) can be
applied again.

If \(D_1\) is the average deficit after the first application, then the
average deficit of the resulting family is at most
\(mD_1+2(m-1)\).  Apply \eqref{eq:recombination} a second time, multiply the
result by \(\theta/m\), and add it to the inequality from the first
application.  The terms containing \(D_1\) cancel.  If \(D_2\) is the final
average deficit, we obtain
\[
 (1-\theta)\left(1+\frac{\theta}{m}\right)M
 \leqslant
 D_0+K+\frac{\theta}{m}\{2(m-1)+K\}
 -\frac{\theta^2}{m}D_2.
\]
Since \(D_2\geqslant0\), it follows that
\begin{equation}\label{eq:repeated-estimate}
        \frac{17255}{22871}M
        \leqslant D_0+\frac{2068995}{251581}.
\end{equation}
Again, the same estimate holds for surpluses.

Put
\[
        C_0=\frac{374167}{20000}.
\]
Consequently, the two-expander estimate gives \(M\leqslant C_0\) whenever
\(D_0\leqslant d_{\mathrm D}\), and the repeated-\(E_3\) estimate does so
whenever \(D_0\leqslant d_{\mathrm I}\), where
\begin{align}
 d_{\mathrm D}
 &=
 C_0\frac{17}{21}-\frac{185}{21}
 =\frac{2660839}{420000},\label{eq:dD}\\
 d_{\mathrm I}
 &=
 C_0\frac{17255}{22871}
 -\frac{2068995}{251581}
 =\frac{5927773487}{1006324000}.\label{eq:dI}
\end{align}

We shall also use the cancellation mentioned in the introduction.  Suppose
that intersections of mean levels \(a\) and \(b\) have been taken in
\(\cA\) and \(\cB\), respectively, and put
\[
        x=aq,\qquad y=bp.
\]
Their average deficit and surplus are at most
\[
 x(1-u)+2(a-1),\qquad y(1+u)+2(b-1),
\]
respectively.  Apply the same expander estimate to the two families, multiply
the deficit inequality by \(y\) and the surplus inequality by \(x\), and add.
After division by \(x+y\), the terms containing \(u\) disappear.  The
effective source cost is
\begin{equation}\label{eq:W}
 W(a,b;q)=
 \frac{y\{x+2(a-1)\}+x\{y+2(b-1)\}}{x+y}.
\end{equation}
In particular, either estimate proves \(M\leqslant C_0\) as soon as \(W\) is
at most the corresponding number in \eqref{eq:dD} or \eqref{eq:dI}.

For \(0<t<1\) and \(t^4\leqslant B\), write
\begin{equation}\label{eq:mB}
 m_B(t)=
 \begin{cases}
 3,&t^3\leqslant B,\\[2mm]
 \displaystyle 3+\frac{t^3-B}{t^3-t^4},&t^4\leqslant B<t^3.
 \end{cases}
\end{equation}
This is the mean level of a triple intersection, or of a convex combination
of triple and fourfold intersections, whose point frequency is at most \(B\).

We need one elementary monotonicity fact.  Take \(a=3\),
\(b=m_B(p)\), \(p=1-q\), and \(1/2\leqslant p\leqslant11/20\).  Direct
differentiation gives
\begin{equation}\label{eq:pure-derivative}
 \frac{d}{dp}W(3,m_B(p);1-p)
 =
 -\frac{6R_B(p)}
 {p^2(B+2p^3-3p^2)^2},
\end{equation}
where
\[
\begin{split}
 R_B(p)={}&B^2(p^2+1)
 +Bp^2(15p^4-32p^3+12p^2+10p-9)\\
 &+p^6(12p^3-41p^2+42p-11).
\end{split}
\]
The two polynomials in parentheses are increasing on
\([1/2,11/20]\), and on that interval they are bounded above by
\(-19/5\) and \(17/10\), respectively.  It follows that
\[
 R_B(p)\leqslant
 \frac{521}{400}B^2-\frac{19}{20}B
 +\frac{17}{10}\left(\frac{11}{20}\right)^6.
\]
For \(B=\beta\) and \(B=\gamma\), the right-hand side is, respectively,
\[
 -\frac{9947816623}{282240000000}
 \quad\hbox{and}\quad
 -\frac{31854253}{640000000}.
\]
Thus the derivative in \eqref{eq:pure-derivative} is positive: in these two
instances \(W(3,m_B(1-q);q)\) decreases as \(q\) increases.

Set
\[
        q_-=\frac{2251}{5000},
        \qquad
        q_+=\frac{47193}{100000}.
\]

\medskip
\noindent\emph{First range: \(0<q\leqslant q_-\).}
Take the convex combination of double and triple intersections of \(\cA\)
for which
\[
 \tau_-=\frac{q_-^2-\gamma}{q_-^2-q_-^3}
       =\frac{10022505000}{13929185749},
 \qquad
 m_-=2+\tau_-=\frac{37880876498}{13929185749}.
\]
The function \(t\mapsto(1-\tau_-)t^2+\tau_-t^3\) is increasing on
\([0,1]\), and its value at \(q_-\) is \(\gamma\).  Hence the resulting
family has point frequencies at most \(\gamma\).  Its average deficit satisfies
\[
\begin{split}
 D_0
 &\leqslant 2m_-q+2(m_--1)\\
 &\leqslant 2m_-q_-+2(m_--1)
 =\frac{102514153370999}{17411482186250}
 <\frac{736}{125}<d_{\mathrm I}.
\end{split}
\]
The repeated-\(E_3\) estimate proves \(M\leqslant C_0\) in this range.

\medskip
\noindent\emph{Second range: \(q_-\leqslant q\leqslant q_+\).}
Use the two-expander estimate and put
\[
        a=m_\beta(q),\qquad b=m_\beta(p).
\]
These choices are legitimate: \(p^3\geqslant1/8>\beta\),
\(p^4\leqslant(11/20)^4<\beta\), and, when \(q^3>\beta\),
\(q^4\leqslant q_+^4<\beta\).

If \(q^3\leqslant\beta\), then \(a=3\).
Equation~\eqref{eq:pure-derivative} shows that \(W\) decreases with \(q\).
Its maximum is therefore attained at \(q_-\), and direct substitution in
\eqref{eq:W} gives
\begin{equation}\label{eq:middle-pure}
 W(a,b;q)
 \leqslant W\bigl(3,m_\beta(1-q_-);q_-\bigr)
 <\frac{1267}{200}<d_{\mathrm D}.
\end{equation}

Suppose that \(q^3>\beta\), so both \(a\) and \(b\) come from the second
line of \eqref{eq:mB}.  Put \(s=q(1-q)\).  On simplifying
\eqref{eq:W}, we find
\begin{equation}\label{eq:W-beta}
 W_\beta(s)=
 \frac{2\{\beta^2s+\beta^2+6\beta s^3+7\beta s^2-\beta s-\beta
       +9s^5+12s^4+3s^3\}}
 {s^2(s^2+s-\beta)}.
\end{equation}
Its derivative is
\begin{equation}\label{eq:W-beta-derivative}
 W_\beta'(s)=
 \frac{2T_\beta(s)}{s^3(s^2+s-\beta)^2},
\end{equation}
where
\[
\begin{split}
 T_\beta(s)={}&\beta^3(s+2)
 -\beta^2(9s^3+6s^2+4s+2)\\
 &+\beta(-33s^5-38s^4-7s^3+6s^2+3s)
 +9s^5(s+1)^2.
\end{split}
\]
Here \(9/40\leqslant s\leqslant1/4\), and
\[
\begin{split}
T_\beta(s)\geqslant{}&
9\left(\frac9{40}\right)^5\left(\frac{49}{40}\right)^2\\
&+\beta\left\{
6\left(\frac9{40}\right)^2+3\left(\frac9{40}\right)
-33\left(\frac14\right)^5-38\left(\frac14\right)^4
-7\left(\frac14\right)^3\right\}\\
&-\beta^2\left\{
9\left(\frac14\right)^3+6\left(\frac14\right)^2
+4\left(\frac14\right)+2\right\}\\
={}&\frac{333130158209}{8028160000000}>0.
\end{split}
\]
The denominator in \eqref{eq:W-beta-derivative} is positive.  Thus
\(W_\beta\) increases with \(s\), while \(s=q(1-q)\) increases with \(q\)
on the present interval.  At the right endpoint, rational
cross-multiplication gives
\begin{equation}\label{eq:middle-fourfold}
        W_\beta\bigl(q_+(1-q_+)\bigr)
        <\frac{63353}{10000}<d_{\mathrm D}.
\end{equation}
Together, \eqref{eq:middle-pure} and \eqref{eq:middle-fourfold} settle the
second range.

\medskip
\noindent\emph{Third range: \(q_+\leqslant q\leqslant1/2\).}
Use the repeated-\(E_3\) estimate and put
\[
        a=m_\gamma(q),\qquad b=m_\gamma(p).
\]
The required convex combinations are available:
\(p^3\geqslant1/8>\gamma\),
\(p^4\leqslant(53/100)^4<\gamma\), and, when needed,
\(q^4\leqslant1/16<\gamma\).

If \(q^3\leqslant\gamma\), then \(a=3\), and
\eqref{eq:pure-derivative} places the maximum at \(q_+\).  Substitution in
\eqref{eq:W} gives
\begin{equation}\label{eq:high-pure}
 W(a,b;q)
 \leqslant W\bigl(3,m_\gamma(1-q_+);q_+\bigr)
 <\frac{147263}{25000}<d_{\mathrm I}.
\end{equation}

It remains to treat \(q^3>\gamma\).  Both \(a\) and \(b\) now come from the
second line of \eqref{eq:mB}.  With \(s=qp\), substitution in
\eqref{eq:W} gives
\[
 W(a,b;q)=
 \frac{2\{\gamma^2s+\gamma^2+6\gamma s^3+7\gamma s^2-\gamma s-\gamma
       +9s^5+12s^4+3s^3\}}
 {s^2(s^2+s-\gamma)}.
\]
A direct factorisation yields
\begin{equation}\label{eq:gamma-factorisation}
 \frac{21-80\gamma}{2}-W(a,b;q)
 =
 \frac{(1-4s)R_\gamma(s)}
 {2s^2(s^2+s-\gamma)},
\end{equation}
where
\[
 R_\gamma(s)=9s^4+(9+20\gamma)s^3+31\gamma s^2
 +20\gamma(1-\gamma)s+4\gamma(1-\gamma).
\]
All coefficients of \(R_\gamma\) are positive.  Moreover,
\(9/40\leqslant s\leqslant1/4\), so \(s^2+s-\gamma>0\) and
\(1-4s\geqslant0\).  It follows that
\[
        W(a,b;q)\leqslant\frac{21-80\gamma}{2}
        =\frac{28}{5}<d_{\mathrm I}.
\]

The three ranges cover \(0<q\leqslant1/2\).  We have proved
\(M\leqslant C_0\) on every finite power-set algebra in the normalisation
\(\Delta=1\).  Lemma~\ref{lem:finite-reduction} and scaling complete the
proof.

\section*{Acknowledgements}

The second-named author gratefully acknowledges support from NCN Sonata-Bis~13
(2023/50/E/ST1/00067).

ChatGPT was used at an exploratory stage to generate ideas that led to the
preliminary bound \(20.5\); the final bound was obtained by the authors.  We
thank Marcin Guzik for helpful discussions and Tomasz Kochanek for bringing the
problem to our attention.


\begin{thebibliography}{99}

\bibitem{AlonSpencer}
N.~Alon and J.~H. Spencer,
\emph{The Probabilistic Method}, 4th ed.,
Wiley Series in Discrete Mathematics and Optimization, Wiley, Hoboken, 2016.
\url{https://doi.org/10.1002/9781119061953}

\bibitem{BondarenkoPrymakRadchenko}
A.~V. Bondarenko, A.~Prymak and D.~Radchenko,
\emph{On concentrators and related approximation constants},
J. Math. Anal. Appl. \textbf{402} (2013), no.~1, 234--241.
\url{https://doi.org/10.1016/j.jmaa.2013.01.019}

\bibitem{FeigeFeldmanTalgamCohen}
U.~Feige, M.~Feldman and I.~Talgam-Cohen,
\emph{Approximate modularity revisited},
SIAM J. Comput. \textbf{49} (2020), no.~1, 67--97.
\url{https://doi.org/10.1137/18M1173873}

\bibitem{GnacikGuzikKania}
M.~Gnacik, M.~Guzik and T.~Kania,
\emph{Approximate modularity: Kalton's constant is not smaller than 3},
Proc. Amer. Math. Soc. \textbf{149} (2021), no.~2, 661--669.
\url{https://doi.org/10.1090/proc/15195}

\bibitem{KaltonRoberts}
N.~J. Kalton and J.~W. Roberts,
\emph{Uniformly exhaustive submeasures and nearly additive set functions},
Trans. Amer. Math. Soc. \textbf{278} (1983), no.~2, 803--816.
\url{https://doi.org/10.1090/S0002-9947-1983-0701524-4}

\bibitem{Pawlik}
B.~Pawlik,
\emph{Approximately additive set functions},
Colloq. Math. \textbf{54} (1987), no.~1, 163--164.
\url{https://doi.org/10.4064/cm-54-1-163-164}

\bibitem{Pippenger}
N.~Pippenger,
\emph{Superconcentrators},
SIAM J. Comput. \textbf{6} (1977), no.~2, 298--304.
\url{https://doi.org/10.1137/0206022}

\end{thebibliography}
\end{document}